\documentclass[11pt]{article}

\usepackage{amsmath}
\usepackage{amsthm}
\usepackage{amssymb}
\usepackage{comment}
\usepackage{graphics}

\def\R{\mathbb{R}}
\def\Rp{\mathbb{R}^+}

\def\N{\mathbb{N}}

\def\S{\mathbb{S}}
\def\D{\mathbb{D}}

\def\P{\mathcal{P}}
\def\Int{\text{{\rm Int}}}
\def\Ext{\text{{\rm Ext}}}
\def \Con {\mathop {\hbox{\rm Con}}}
\def\sm{\setminus}

\newtheorem{theorem}{Theorem}[section]
\newtheorem*{theorem*}{Theorem}
\newtheorem{lemma}[theorem]{Lemma}
\newtheorem{proposition}[theorem]{Proposition}

\newtheorem*{conjecture*}{Conjecture}

\theoremstyle{remark}
\newtheorem{remark}[theorem]{Remark}
\theoremstyle{remark}
\newtheorem*{remark*}{Remark}

\begin{document}

\title{Checking atomicity of conformal ending
measures \\  for Kleinian groups}

\author{
{\em Kurt Falk}\footnote{Research supported by the
Science Foundation Ireland}\\
\footnotesize{{\sl Department of Mathematics, NUI Maynooth,}}\\
\footnotesize{{\sl Co. Kildare, Ireland}}\\
\footnotesize{{\tt kfalk@maths.nuim.ie}}\bigskip\\
{\em Katsuhiko Matsuzaki}\\
\footnotesize{{\sl Department of Mathematics, Okayama University,}}\\
\footnotesize{{\sl Okayama, 700-8530, Japan}}\\
\footnotesize{{\tt matsuzak@math.okayama-u.ac.jp}}\bigskip\\
{\em Bernd O. Stratmann}\\
\footnotesize{{\sl Mathematical Institute, University of St Andrews,}}\\
\footnotesize{{\sl North Haugh, St Andrews KY16 9SS, UK}}\\
\bigskip
\footnotesize{{\tt bos@maths.st-and.ac.uk}}
}

\date{}

\maketitle

\begin{abstract}
\noindent
In this paper we address questions of continuity and atomicity
of conformal ending measures for arbitrary non-elementary Kleinian
groups. We give sufficient conditions under which such ending
measures are purely atomic. Moreover, we will show that if a conformal
ending measure has an atom which is contained in
the big horospherical limit set, then this atom has to be a
parabolic fixed point. Also, we give detailed discussions of
non-trivial examples for purely atomic as well as for non-atomic
conformal ending measures.
\end{abstract}
\bigskip



\section{Introduction and statement of the results}
The idea of a conformal ending measure was originally
 sketched  briefly by Sullivan in \cite{Sul} for a Kleinian group
with parabolic elements.
Subsequently, this construction was investigated more
rigorously and in greater generality in \cite{AFT}.
The construction of a conformal ending measure is
very similar to the well known Patterson measure construction
(\cite{Pat}). However, in \cite{FT}  it was shown that there are cases
in which conformal ending measures can not be derived via Patterson's
construction.
The basic idea of a conformal ending measure is as follows (see Section
\ref{alt-constr} for a more detailed review of conformal ending measures).
Let $\left(z_{n}\right)$ be an {\em ending sequence}, that is, a
sequence of points in the Poincar\'e ball model $(\D,d)$
of the hyperbolic space which converges
radially to some point $\zeta \in \S=\partial \D$ within some
Dirichlet fundamental domain of a Kleinian group $G$.
If $\zeta$ is a limit point of $G$, then
the orbit $G(z_n)$ accumulates at the limit set $L(G)$, for $n$ tending to
infinity.
Here, convergence is meant with respect to the
Hausdorff metric on closed subsets of the closure
$\overline \D=\D \cup \S$ in the Euclidean topology.
If $s$ is chosen such that the Poincar\'e series
$\sum_{g \in G} j(g,z)^s$ of $G$ at $z \in \D$ with exponent $s$ converges, then there exist purely atomic
$s$-conformal probability measures $\mu_n$ having their atoms in $G(z_n)$.
Here, $j(g,z)$ denotes the conformal derivative of
$g$ at $z \in \overline \D$,
that is, the uniquely determined positive number such that $g'(z)/j(g,z)$
is orthogonal, where $g'(z)$ is the Jacobian of $g$ at $z$.
For $\zeta \in \S$ we have
$j(g,\zeta)= (1-|g^{-1}(0)|^2)/|\zeta-g^{-1}(0)|^2$
(see Section~\ref{purely-atomic} for further details).
Consequently, weak limits of these measures $\mu_n$ give
rise to $s$-conformal measures supported on $\overline \D$,
and these measures are called
{\em $s$-conformal ending measures}.
The aim of this paper is to address questions of continuity and atomicity
for this type of ending measures.
Throughout we will use the phrase {\em essential support} of a measure
to denote any measurable set of full measure (not
necessarily closed nor uniquely determined), and ${\bf 1}_{z}$ denotes the Dirac delta at $z$.
The following statement  summarises Theorem~\ref{basic1},
Lemma~\ref{bounded-parabolic}, Proposition~\ref{pointmass}
and Theorem~\ref{general}.

\begin{itemize}
    \item[]
{\it Let $\left(z_{n}\right)$ be a given ending sequence
converging to $\zeta \in \S$, and
assume that $j(h, \zeta) = 1$ for
every element $h$ of the stabiliser $\mathrm{Stab}_G(\zeta)$ of
$\zeta$. Moreover, assume that the reduced horospherical
Poincar\'e series $\sum_{g \in G/\mathrm{Stab}_G(\zeta)}j(g,\zeta)^s$
of $G$ at $\zeta$ with exponent $s$ converges.
%% for $s$ greater than or equal to the exponent of
%% convergence $\delta(G)$ of the Poincar\'e series of $G$.
We then have that the essential support of the associated
$s$-conformal ending measure $\mu$ is equal to $G(\zeta)$. In
particular,
$\mu$ is purely atomic and coincides with the measure $\mu_\zeta$
given by
$$
\mu_\zeta=\frac{\sum_{g \in G/\mathrm{Stab}_G(\zeta)} j(g,\zeta)^s
{\bf 1}_{g(\zeta)}}{\sum_{g \in G/\mathrm{Stab}_G(\zeta)} j(g,\zeta)^s}.
$$
If $\zeta$ is an ordinary point of $G$ or a bounded parabolic fixed
point of $G$, then the reduced horospherical Poincar\'e series of $G$
at $\zeta$ with exponent $s$ converges whenever the Poincar\'e series
$\sum_{g \in G} j(g,z)$, $z \in \D$, converges.}
\end{itemize}
This result provides us with typical examples of purely atomic conformal
ending measures.  Another class of examples is
obtained by using a result of \cite{AFT},  stating that
a conformal ending measure associated to a certain given end
%% (in the sense of \cite{AFT})
is essentially supported on the associated end limit set. 
We then have that if this
end limit set is countable, then every ending measure  essentially
supported on it is necessarily purely atomic. Besides, this raises
the question whether it is possible to have a countable end limit set
associated to an end which does not originate from a parabolic fixed point.
In fact, a result of Bishop \cite{bi}, stating that a
Kleinian group acting on three-dimensional hyperbolic space is
geometrically finite if and only if the set of non-radial limit points is
countable, strongly suggests that the case of parabolic cusps is the
only case in which the end limit set of a given end is countable.

In Section~\ref{big-horo-atoms} we first show that if a conformal
ending measure for a Kleinian group $G$ has an atom which is contained in
the big horospherical limit set $L_H(G)$ (see Section \ref{pre} for the
definition), then this atom has to be a
parabolic fixed point.
Note, above we gave a sufficient condition, in terms
of the convergence of the reduced horospherical Poincar\'e series at
$\zeta$, under which a conformal ending measure is purely atomic.
However, unless $\zeta$ is an ordinary point
or $\zeta$ is a bounded parabolic fixed point, it is usually
not an easy task to verify this convergence condition.
Hence, it is desirable to find
conditions which allow to decide more easily whether a conformal ending
measure has atoms or not.
By employing a result of Tukia \cite{Tu}, which states that there
always exists a measurable fundamental set for the
action of  a Kleinian group $G$ on the
complement of $L_H(G)$, in Section~\ref{big-horo-atoms} we show
that if  $G$ admits an $s$-conformal measure
essentially supported on the dissipative part of the action of
$G$ on $\S$, then we always have that there exists a purely atomic $s$-conformal ending measure for
$G$. This observation then gives rise to the following result (see
Theorem  \ref{thm5}). Here, $\delta(G)$ denotes the abzissa of
convergence of the Poincar\'e series of $G$.
\begin{itemize} \item[]
    {\it
For $s\geq \delta(G)$, we
have that
 every $s$-conformal ending measure for $G$ supported on
$L(G)$ is essentially supported on $L_H(G)$
if and only if every $s$-conformal ending measure for $G$ supported
on $L(G)$ can have atoms only at parabolic fixed points of $G$.}
\end{itemize}
Eventually, in Section~\ref{examples}, we give various non-trivial 
examples of purely atomic and of non-atomic
conformal ending measures. These include examples of
infinitely generated Schottky groups whose conformal ending
measures have atoms at J{\o}rgensen points, and examples of
non-atomic conformal ending measures on limit sets of
normal subgroups of finitely generated Schottky groups.
Moreover, we give the construction of a Kleinian group $\Gamma$
which admits a purely atomic $\delta(\Gamma)$-conformal ending measure
at some J{\o}rgensen point which is not a parabolic fixed point of $\Gamma$.

\section{Preliminaries}\label{pre}
Hyperbolic geometry in the $(N+1)$-dimensional ball model
$(\D,d)$, $N \geq 2$, and conformal geometry on the
boundary sphere $\S$ have the same automorphism group,
namely the group $\Con(N)$ of conformal maps in $\S$ which is
isomorphic to
the group of orientation-preserving hyperbolic isometries of $\D$.
(For $N=1$, $\Con(1)$ is merely defined by the latter group.)
Any transformation in $\Con(N)$ gives rise to an isometry of
$\D$, and vice versa.
It is well known that this isomorphism
arises naturally from the principle of Poincar\'e extension,
based on the elementary observation that a $(N+1)$-dimensional
hyperbolic half-space $H$ in $\D$ extends in a unique
way to a $N$-dimensional ball $\overline{H} \cap \S$.
A Kleinian group $G$ is a discrete subgroup of $\Con (N)$
and its limit set $L(G)$  is the
derived set of an arbitrary $G$-orbit, that is
$L(G) = \overline{G(z)} \setminus G(z)$, for some arbitrary $z \in \D$.
It is well known that $L(G)$ is equal to the union of the set
$L_r(G)$ of {\em radial limit points} (also called conical limit points)
and the set $L_t(G) = L(G) \setminus L_r(G)$ of {\em transient limit points}
(also called escaping limit points), where
\[
L_r(G) = \left\{ \xi \in L(G) : \liminf_{T \to \infty} \Delta(\xi_T)< \infty
\right\}.
\]
Here, $\xi_T$ denotes  the unique point on
the hyperbolic geodesic ray from the origin $0 \in \D$ to $\xi$ for
which $d(0, \xi_T)=T$,
and $\Delta(z)= \inf \{ d(z,g(0)) : g \in G \}$ denotes the hyperbolic distance
of $z \in \D$ to the orbit $G(0)$.
Important subsets of $L(G)$ are the set $L_{ur}(G)$ of
{\em uniformly radial limit points} and the set $L_J(G)$ of {\em J{\o}rgensen
limit points}. These are given as follows (cf. \cite{FS}, \cite{past01}).
$L_{ur}(G)$ is the set of  $\xi \in L(G)$ for which
$\limsup_{T \to \infty} \Delta(\xi_T) < \infty$, and
$L_J(G)$ is the set of  $\xi \in L(G)$ for which
there exists a geodesic ray towards $\xi$ which is fully contained in
some Dirichlet fundamental domain of $G$.
One easily verifies that $L_{ur}(G) \subseteq  L_r(G)$ and
$L_J(G) \subseteq L_t(G)$.
Next, recall that an element of $\Con (N)$ is called {\em parabolic}
if it has precisely one fixed point in $\S$. If
a Kleinian group $G$ contains parabolic elements, then the fixed
points of these elements
are contained in $L_J(G)$.
A horosphere at $\zeta \in \S$ can be represented by
the equivalued plane
$h_\zeta(c)=\{z \in \D : k(z, \zeta)=c\}$ of the Poisson kernel
$k(z, \zeta)=(1-|z|^2)/|\zeta-z|^2$,
for some constant $c>0$. Clearly, $h_\zeta(c)$ is a Euclidean
sphere tangent to $\S$ at $\zeta$ of radius equal to $1/(1+c)$.
The associated horoball $H_\zeta(c)$ is the open ball bounded by this horosphere, that is,
$$
H_\zeta(c)=\{z \in \D : k(z, \zeta) > c \}.
$$
Let $\mathrm{Stab}_G(\zeta)$ denote the stabiliser of 
$\zeta \in \overline{\D}$ in $G$ consisting of all
elements of $G$ fixing $\zeta$.
Clearly, if  $\zeta$ is the fixed point of a parabolic element
of $G$, then we have that any horoball at $\zeta$ is invariant
under $\mathrm{Stab}_G(\zeta)$.
Moreover, there is a constant $c>0$ such that
$g(H_\zeta(c)) \cap H_\zeta(c)=\emptyset$, for every
$g \in G \setminus \mathrm{Stab}_G(\zeta)$.
A cusped region for $\zeta$ is given by
the quotient $H_\zeta(c)/\mathrm{Stab}_G(\zeta)$ which we assume to be embedded in the hyperbolic orbifold $M_G=\D/G$.

The {\em convex hull} $C(L(G))$ is the smallest closed convex subset
of $\D$ whose Euclidean closure in $\overline{\D}$ contains $L(G)$.
Equivalently, $C(L(G))$ is the
hyperbolic convex hull of the union of all geodesic lines with end
points in $L(G)$. Clearly, the set $C(L(G))$ is invariant
under $G$, and hence, $C_G=C(L(G))/G$ is a closed subset of $M_G=\D/G$.
The set $C_G$ is called the {\em convex core} of $G$ or of $M_G$.
A Kleinian group $G$ is called {\em geometrically finite} if
there exists $\varepsilon>0$ such that the hyperbolic
$\varepsilon$-neighbourhood
of $C_G$ has finite hyperbolic volume.
We say that $\zeta$ is a {\em bounded parabolic fixed point} of
$G$ if the  intersection of  $C_G$ and
$h_\zeta(c)/\mathrm{Stab}_G(\zeta)$  is compact.
A theorem of Beardon and Maskit \cite{BeMa}
asserts that $G$ is geometrically finite if and only if
the limit set $L(G)$ is equal to the union of $L_r(G)$
and the set of bounded parabolic fixed points.

Let us also recall the following well known basic formulae 
(see e.g. \cite{Ni}).
Let $j(g,z)$ denote the {\em conformal derivative} of
$g \in \Con(N)$ at $z \in \overline \D$, that is,
$$
j(g,z)=\frac{1-|g^{-1}(0)|^2}{|z^*-g^{-1}(0)|^2|z|^2},
$$
where $z^*$ denotes the image of $z$
under the reflection at $\S$. An elementary calculation gives that if
$\zeta \in \S$, then
$$
j(g,\zeta)=\frac{1-|g^{-1}(0)|^2}{|\zeta-g^{-1}(0)|^2},
$$
and hence, $j(g,\zeta)=k(g^{-1}(0), \zeta)$.
Also, if $z \in \D$, then one finds
$$
  j(g,z) = \frac{1 - |g(z)|^{2}}{ 1 - |z|^{2}}.
$$
For $z_1,z_2 \in \D$, we define $d_\zeta(z_1,z_2)$ to be the {\em signed
hyperbolic distance} between the horosphere at $\zeta$ through $z_1$
and the horosphere at $\zeta$ through $z_2$, where `signed'  refers to
that $d_\zeta(z_1,z_2)$
is positive if $z_2$ is contained in the horoball bounded by the
horosphere at $\zeta$ through $z_1$, and negative otherwise.
Clearly, $|d_\zeta(z_1,z_2)| \leq d(z_1,z_2)$.
It is well known that the relation between this horospherical distance and the Poisson kernel is given by
$$
d_{\zeta}(z,0)= -\log k(z,\zeta).
$$

The {\em Poincar\'e series} associated with a Kleinian group $G$ is defined
for $z \in \D$ and $ s \in \R$ by
$$
P(z,s) = \sum_{g \in G} j(g,z)^{s}.
$$
Note that convergence and divergence of
$P(z,s)$ does not depend on the choice of $z \in \D$.
The {\em exponent of convergence} of $G$ is then given by
$$
\delta(G) = \inf \{ s \in \Rp : P(z,s) < \infty\}.
$$
If $P(z,s)$ converges at $s = \delta(G)$,
that is if $P(z,\delta(G)) < \infty$,
then $G$ is said to be of {\em convergence type}. Otherwise,
$G$ is said to be of {\em divergence type}.
Note that in \cite{BJ} it was shown if $G$ is some
arbitrary non-elementary Kleinian group,
then the Hausdorff dimension of $L_{r}(G)$ and of $L_{ur}(G)$ are both
equal to $\delta(G)$ (see
also \cite{SBJ}).

Also, recall that a Borel probability measure $\mu$ on $\overline \D$
is called {\em $s$-conformal} for a Kleinian group $G$ if, for each $g
\in G$ and any measurable set $A$ on
$\overline \D$,
$$
\mu(g(A))=\int_A j(g,z)^s\, d\mu(z).
$$
This condition is equivalent to that for each continuous function $f$ on $\overline \D$ and
for each $g \in G$, we have
$$
\int_{\overline \D}f(g^{-1}(z))\,d\mu(z)
=\int_{\overline \D}f(z)j(g,z)^s\,d\mu(z).
$$
For a Kleinian group $G$, the {\em big horospherical limit set} $L_H(G)$ consists of
all points $\zeta \in L(G)$ with the property that there is a constant
$c>0$ such that the orbit $G(0)$ accumulates at $\zeta$ within the
horoball $H_\zeta(c)$. It is easy to see that $L_r(G) \subseteq
L_H(G)$ and that
parabolic fixed points of $G$ are contained in $L_H(G)$.

It is known that a Kleinian group $G$ acts {\em conservatively} on $L_H(G)$  with respect to
any conformal measure $\mu$ for $G$ (see Tukia \cite{Tu}).
This means that there are infinitely many elements $g \in G$ such that
$\mu(A \cap g(A))>0$,
for each measurable set $A\subset L_H(G)$ for which $\mu(A)>0$. On the other hand, Tukia \cite{Tu} also showed
that $G$ acts {\em dissipatively} on the complement $\S \sm L_H(G)$, that
is, there exists a measurable fundamental set $E \subset \S \sm L_H(G)$
such that $G(\zeta) \cap E$ is a singleton, for
each $\zeta \in \S \sm L_H(G)$.
These results provide significant generalisations of a theorem of Sullivan
\cite{Sul81}, who derived these assertions  for the special case in
which $\mu$ is equal to the $N$-dimensional Lebesgue measure $m$.
Next, recall that the {\em big Dirichlet set} $D(G)$ of $G$ is equal
to the boundary at infinity of some Dirichlet
fundamental domain.  Clearly, $D(G)$ can be written as a disjoint
union as follows:
$$
D(G) = \Omega(G) \cup L_J(G).
$$
It is known that $D(G) \cup L_H(G)=\S$ and that $m(D(G) \cap L_H(G))=0$
(see \cite{MT}).

Finally, let us also recall from \cite{AFT} that a connected open 
subset $E$ of $M_G = \D/G$ is called an {\em end of $M_G$},
if $\partial E$ is compact and non-empty and if $E$ has non-compact
closure in $M_G$.  For a suitable index set $I$, let
$\{E^0_i:i \in I\}$ denote the set of all connected components
of the lift of $E$ to $\D$, and let $G_i$ denote
the stabiliser $\mathrm{Stab}_G(E^0_i)$ of
$E^0_i$ in $G$. The set $E^0_i$ is called an {\em end of $G$}, and $G_i$ the
corresponding {\em end group of $G$}. By construction,
the $G_i$ are subgroups of $G$ which are all conjugate to each other.
A point $z \in \S$ is called an {\em endpoint of an end} $E^0_i$ of $G$ if,
whenever $\beta$ is a hyperbolic ray towards $z$, there exists a
subray $\beta'$ of $\beta$ that is contained in $E^0_i$, for which
the hyperbolic distance $d(x,\partial E^0_i)$ tends to infinity as
$x$ approaches $z$ on $\beta'$.
Note that the set of endpoints of $E^0_i$ has empty intersection 
with $L_r(G)$.
If additionally $z \in L(G_i)$, then $z$ is called  {\em end limit point} of
$E^0_i$. Note that the set of end limit points of $E^0_i$ is
contained in $L(G_i) \setminus L_r(G)$.
The union of all end limit points of all ends $E^0_i$, $i \in I$, shall be referred to as the {\em end limit set of $G$ associated with $E$}, or simply as the {\em end limit set of $E$}.

So far, ends have been subsets $E$ of $M_G$ (or of the lift to $\D$ of
such a set) such that $E$ is a non-compact component of
$M_G \setminus \partial E$, where the boundary $\partial E$ of $E$ in
$M_G$ is compact.
Note that the definition of an end can also be given in terms
of $\overline{M}_G = (\D \cup \Omega (G))/G$. Namely, in these terms,
$E$ is a component of $\overline{M}_G \setminus \bar \partial E$,
where the boundary $\bar \partial$ is taken in $\overline{M}_G$.
In \cite{AFT} such a set $E$ was called an `end with boundary'.
Note that for ends with boundary one has by \cite[Theorem~4.4]{AFT} that there exists a non-trivial conformal measure supported by the end limit
point set of an end with boundary. Moreover,  \cite[Proposition~4.5]{AFT} then allows to extend this conformal measure for an end
group to a conformal measure for $G$. We shall make use of this
observation in the proof of Proposition~\ref{mut-sing} of this paper.
For further details on end limit sets and ends with boundary we refer
to \cite{AFT} and \cite{FT}.

\section{The construction of conformal ending measures}
\label{alt-constr}

In order to recall the construction of an $s$-conformal ending
measure,  let $F$ be a Dirichlet fundamental domain of some given Kleinian
group $G$, and let $\overline{F}$ be its  closure in $\overline{\D}$.
A sequence $\left(z_{n}\right)$ of elements $z_{n} \in F$
is called an {\em ending sequence} tending to
$\zeta \in \overline{F} \cap \S$, if each of the $z_{n}$ lies within a bounded distance from some geodesic ray towards $\zeta$ contained in $F$, and if the sequence $\left(z_{n}\right)$ tends to $\zeta$ with respect to the Euclidean metric.
Let $s \in \Rp$ be chosen  such that
$P(z,s)< \infty$. For a given ending sequence $\left(z_{n}\right)$
tending to $\zeta \in \overline{F} \cap \S$,
we define, for each $n \in \N$, the point mass $\mu_n$ on the closed unit
ball $\overline{\D}$ as follows:
$$
\mu_n = \frac{1}{P(z_{n},s)} \sum_{g \in G} j(g,z_n)^s \;
\hbox{\bf 1}_{g(z_n)}.
$$
The Helly-Bray Theorem then
ensures that, by passing to a subsequence if necessary, the sequence
$(\mu_n)$ has a weak limit measure $\mu$, and  we will refer to this
measure as {\em $s$-conformal ending measure}.
By construction, for each of the $\mu_n$ we have,  for any $g \in G$
and for any continuous function
$f$ on $\overline \D$,
\begin{eqnarray*}
\int_{\overline \D}f(g^{-1}(z))\,d\mu_n(z)
&=&\frac{1}{P(z_{n},s)} \sum_{h \in G}f(g^{-1}h(z_n)) j(h,z_n)^s\\
&=&\frac{1}{P(z_{n},s)} \sum_{h \in G}f(h(z_n)) j(gh,z_n)^s\\
&=&\frac{1}{P(z_{n},s)} \sum_{h \in G}f(h(z_n)) j(g,h(z_n))^sj(h,z_n)^s\\
&=&\int_{\overline \D}f(z)j(g,z)^s\,d\mu_n(z).
\end{eqnarray*}
Since $(\mu_n)$ converges weakly to $\mu$,
it therefore follows that $\mu$ satisfies the same $s$-conformal transformation
rule. Therefore, we have, for all $g \in G$ and  $A \subset \overline{\D}$
measurable,
$$
\mu (g(A)) = \int_A j(g,z)^s \, d\mu .
$$

Finally, let us remark that  if $L_J(G)$ is non-empty and if $\mu$ is an
$s$-conformal ending measure derived from an
ending sequence $(z_{n})$ tending towards some $\zeta \in L_J(G)$,
then the support of $\mu$ has to be equal to
$L(G)$. Indeed, consider any open set
$U \subset \overline \D$ whose closure $\overline U$ is
contained in $\overline \D \sm L(G)$.
The discontinuous action of $G$ on $\overline \D \sm L(G)$
immediately implies that there exist at most finitely many $g \in G$
such that $g(F) \cap U \neq \emptyset$.
Furthermore, we have that $g(F) \cap \overline U$ is compact in
$\overline \D \sm L(G)$, for each $g \in G$.
Since the sequence $(z_n)$ used in the definition of $\mu_n$
was chosen to be an ending sequence contained in $F$, it then follows that
there exists $n_U \in \N$ such that $\mu_n(U) = 0$, for all $n \geq n_U$.
This shows that $\mu(U) = 0$,
and hence, $\mu$ is  supported on $L(G)$.

\section{Convergence at infinity and purely atomic measures}
\label{purely-atomic}

Throughout, let $G$ be a Kleinian group.
As already  mentioned briefly in the introduction, extensions of the
Poincar\'e  series $P(z,s)$ to points in $\S$ gives rise to what we call
the {\em horospherical Poincar\'e series} of $G$.
More precisely,  the horospherical
Poincar\'e series $\P(\zeta,s)$ of $G$ at $\zeta \in \S$ with
exponent $s\in \Rp$ is given by
$$
\P(\zeta,s)= \sum_{g \in G} j(g,\zeta)^s.
$$
Note that Morosawa \cite{Mo} also considered this
type of horospherical Poincar\'e series
and obtained basic results concerning atoms
of conformal measures, some of
which will be generalised in this paper.
Also, we refer to  \cite{Ni} (Section~3.5) for some basic facts concerning
the atomic part of a conformal measure.

Let us first prove the following two elementary lemmata. These will turn out to be useful later.

\begin{lemma}
\label{convergence1}
If $\P(\zeta,s) <\infty$ for some $\zeta \in \S$,
then $ P(z,s) <\infty$ for every $z \in \D$.
Moreover, if $P(z,s) <\infty$ for some (and hence for all) $z \in \D$, then
$\P(\zeta,s) <\infty$ for each $\zeta \in \Omega(G)$.
In particular, we hence have that if $\P(\zeta,s) <\infty$ for some $\zeta \in \S$, then $s \geq
\delta(G)$.
\end{lemma}
\begin{proof}
Using the identities for the conformal derivative $j(g,\cdot)$ stated
in Section 2, one immediately verifies that for $\zeta \in \S$ and
$z \in \D$ we have
$$
j(g,\zeta)
=\frac{|z^*-g^{-1}(0)|^2|z|^2}{|\zeta-g^{-1}(0)|^2}j(g,z)
\geq \frac{(1-|z|)^2}{4} j(g,z).
$$
This gives the first statement in the lemma.
%% (Alternatively, the inequality
%% $e^{-sd(g^{-1}(0),0)} \leq e^{-sd_{\zeta}(g^{-1}(0),0)}$
%% yields $P(0,s)<\infty$.)
For the second statement note that if $\zeta \in \Omega(G)$, then
$|\zeta-g^{-1}(0)|^2 \geq c$, for some
constant $c>0$ not depending on $g \in G$. Hence,
$$
j(g,\zeta) \leq \frac{{(1+|z|)^2}}{c} j(g,z),
$$
which gives the second statement in the lemma.
\end{proof}

\begin{lemma}
\label{convergence2}
Let $\left(z_{n}\right)$ be a sequence in $\D$ which converges to
$\zeta \in \S$. Then the following hold:
\begin{enumerate}
\item [$(\rm i)$]
We always have
$\liminf_{n \to \infty}P(z_n,s) \geq
\P(\zeta,s)$, and
therefore, if $\P(\zeta,s)$ diverges, then
$\lim_{n \to \infty} P(z_n,s)=\infty$.
\item [$(\rm ii)$]
If $\zeta \in \Omega(G)$, then
$\lim_{n \to \infty}P(z_n,s) =
\P(\zeta,s)$.
\item [$(\rm iii)$]
If $\left(z_{n}\right)$ converges radially to $\zeta \in \S$, then
$\lim_{n \to \infty}P(z_n,s) =
\P(\zeta,s)$.
\end{enumerate}
\end{lemma}

\begin{proof}
The statement in (i) is an immediate consequence of Fatou's Lemma.
For the remaining statements, we can assume without loss of
generality that $\P(\zeta,s)$ converges.
If $\zeta \in \Omega(G)$, then clearly $|\zeta-g^{-1}(0)|^2$ is uniformly bounded away from $0$, for all $g \in G$. Therefore, since $|z_n^*-g^{-1}(0)|^2$
tends to $|\zeta-g^{-1}(0)|^2$ for
$n$ tending to infinity, there exists a constant $c_{1}>0$ such
that, for each $g \in G$ and every $n\in \N$,
$$
|z_n^*-g^{-1}(0)|^2|z_n|^2 \geq c_{1}|\zeta-g^{-1}(0)|^2.
$$
Hence, $j(g,z_n)^s \leq c_{1}^{-s} j(g,\zeta)^s$.
For the remaining assertion, assume that $z_n$ converges radially to $\zeta$,
in the sense that each $z_n$ is at most a uniformly bounded distance away from some geodesic ray ending at $\zeta$.
This immediately implies that the angle $\theta$ at $\zeta$ in the Euclidean triangle with vertices $g^{-1}(0)$, $\zeta$ and $z_n^*$ is uniformly bounded away from zero. An application of the Euclidean sine rule then gives that $|\zeta-g^{-1}(0)| \leq (1/\sin \theta) |z_n^*-g^{-1}(0)|$, for all
$n\in \N$ and $g \in G$.
Thus, there exists a constant $c_{2}>0$
such that, for each $g \in G$ and every $n\in \N$,
$$
|z_n^*-g^{-1}(0)|^2|z_n|^2 \geq c_{2}|\zeta-g^{-1}(0)|^2.
$$
This implies that $j(g,z_n)^s \leq c_{2}^{-s} j(g,\zeta)^s$.
Since we have by definition that $j(g,z_n)$ tends to $j(g,\zeta)$ for $n$ tending to infinity, the dominated convergence theorem now gives in the situation of (ii) and of (iii) that $P(z_n,s)$ tends to $\P(\zeta,s)$ 
for $n$ tending to infinity.
\end{proof}

The following theorem shows that convergence and divergence of the horospherical Poincar\'e series
is crucial for finding out whether $s$-conformal
ending measures have atoms or not.

\begin{theorem}
\label{basic1}
Let $\left(z_{n}\right)$ be a given ending sequence
which tends to $\zeta$, for some
$\zeta \in D(G)=\Omega(G) \cup L_{J}(G)$ not fixed by
any non-trivial element of $G$. Also, suppose that $\P(\zeta,s)$ converges
for $s \geq \delta(G)$.
We then have that the essential support of the associated
$s$-conformal ending measure $\mu$ is equal to $G(\zeta)$. In
particular, $\mu$ is purely atomic and coincides with
$$
\mu_\zeta=\frac{1}{\P(\zeta,s)}\sum_{g \in G} j(g,\zeta)^s \,
{\bf 1}_{g(\zeta)}.
$$
\end{theorem}

\begin{proof}
Let $K=\{z_{1}, z_{2},\ldots\} \cup \{\zeta\}$ and consider
$$
\mu_{n}=\frac{1}{P(z_{n},s)}\sum_{g \in G}
j(g,z_{n})^s \, {\bf 1}_{g(z_{n})}.
$$
Clearly $K$ is a compact subset of $\D \cup \S$.
Combining this with the observation that
$\limsup_{n \to \infty} \mu_{n}(C) \leq \mu(C)$
for each $C \subset \D \cup \S$ compact, which is an immediate
consequence of weak convergence, we obtain
\begin{eqnarray*}
1 &\geq& \mu(G(\zeta))
= \sum_{h \in G} \mu (h(K))
\geq
\sum_{h \in G} \limsup_{n \to \infty} \mu_{n} (h(K))
\\ &=&
\sum_{h \in G} \limsup_{n \to \infty}
\frac{j(h,z_{n})^{s}}{P(z_{n},s)}.
\end{eqnarray*}
By Lemma \ref{convergence2}, we have
$$
\lim_{n \to \infty}
P(z_{n},s)=\P(\zeta,s).
$$
Since  $\P(\zeta,s)$ converges by assumption,
the last term in the above chain of inequalities is equal to $1$.
This shows that $\mu(G(\zeta))=1$, and thus the probability measure
$\mu$ is essentially supported on the countable set $G(\zeta)$.
\end{proof}

Note that the statement of Theorem~\ref{basic1} includes the case in
which $\zeta \in \Omega(G)$. In this case the resulting measure
$ \mu_\zeta = \frac{1}{\P(\zeta,s)} \sum_{g \in G} j(g,\zeta)^s \,
\hbox{\bf 1}_{g(\zeta)}$ is a conformal measure in the sense we use here
(see Section~\ref{pre}), although strictly speaking it is not supported on the limit set of $G$.

Next we consider the case in which $\zeta$ is a parabolic fixed point
of $G$.
In this situation we have that the horospherical Poincar\'e series
$\P(\zeta,s)$ at $\zeta$
diverges, for every $s\in \R$. This follows, since
$j(g,\zeta)= e^{-d_\zeta(g^{-1}(0),0))}$ is
constant for infinitely many $g \in G$.
In this situation we then consider the {\em reduced horospherical
Poincar\'e series}
$$
\P_{\mbox{\tiny{red}}}(\zeta,s)=\sum_{g \in G/\mathrm{Stab}_G(\zeta)}
j(g,\zeta)^{s},
$$
and if this series converges then we have that the associated
$s$-conformal measure
$$
\mu_\zeta = \frac{1}{\P_{\mbox{\tiny{red}}}(\zeta,s)}
\sum_{g \in G/\mathrm{Stab}_G(\zeta)}
j(g,\zeta)^s \;
\hbox{\bf 1}_{g(\zeta)}
$$
is purely atomic.
Note that $\mu_\zeta$ is indeed an $s$-conformal measure for $G$, as a computation similar to the one in Section~\ref{alt-constr} readily shows.
For the case in which $\zeta$ is a bounded parabolic
fixed point, the following lemma shows
that if the Poincar\'e series converges, then  the reduced horospherical
Poincar\'e series converges as well.
Therefore, if $s >\delta(G)$ or if $s=\delta(G)$ for $G$  of convergence
type, then we always have that there exists a purely atomic
$s$-conformal measure essentially supported on $G(\zeta)$.

\begin{lemma}
\label{bounded-parabolic}
For $\zeta$ a bounded parabolic fixed point of a Kleinian group $G$, we
have that
%%% and let $z \in \D$ be given.
if $P(z,s)$ converges for $s \geq \delta(G)$, then so does
$\P_{\mbox{\tiny{red}}}(\zeta,s)$.
\end{lemma}

\begin{proof}
Let $c>0$ be chosen such that
$H_\zeta(c)/\mathrm{Stab}_G(\zeta)$ is a cusped region in $M_G$.
Without loss of generality, we may assume
$0 \in h_\zeta(c) \cap C(L(G))$.
Since $\zeta$ is a bounded parabolic fixed point,
there exists a constant $b>0$ such that
$d(0,(\mathrm{Stab}_G(\zeta))\,(z)) \leq b$, for each
$z \in h_\zeta(c) \cap C(L(G))$.
Then choose representatives $\gamma \in G$ of the cosets in
$G/\mathrm{Stab}_G(\zeta)$ such that
$$
d(0,\gamma(0))= \min \{d(0,g(0)): g \in \gamma \,
\mathrm{Stab}_G(\zeta)\}.
$$
With this choice, we then have, for each of these coset
representatives $\gamma$,
$$
e^{-d_\zeta(\gamma^{-1}(0),0)} \leq e^b e^{-d(\gamma^{-1}(0),0)}.
$$
Since $j(g,\zeta)=e^{-d_{\zeta}(\gamma^{-1}(0),0)}$
for every $g \in \gamma \, \mathrm{Stab}_G(\zeta)$,
it follows that
%\begin{eqnarray*}
$$
    \P_{\mbox{\tiny{red}}}(\zeta,s)
=
\sum_{\gamma \in G/\mathrm{Stab}_G(\zeta)}
\!\!\! e^{-sd_\zeta(\gamma^{-1}(0),0)}
\leq
e^{sb} \!\!\! \sum_{\gamma \in G/\mathrm{Stab}_G(\zeta)}
\!\!\! e^{-sd(\gamma^{-1}(0),0)}
\leq
e^{sb} P(0,s).
$$
%\end{eqnarray*}
This finishes the proof of the lemma.
\end{proof}

We now look at the more general situation in which there exists a purely atomic $s$-conformal measure essentially supported on $G(\zeta)$ for some point $\zeta \in \S$.
Note that the measures considered in the following proposition  are conformal measure which are not necessarily ending measures as introduced in Section~\ref{alt-constr}.

\begin{proposition}
\label{pointmass}
Let $G$ be a Kleinian group, and let $\zeta \in \S$ and $s \in \R$ be
fixed. Then there exists a purely atomic
$s$-conformal measure $\mu_\zeta$  essentially supported on $G(\zeta)$ if and only if
$$
\hbox{$j(h, \zeta) = 1$
for each $h \in \mathrm{Stab}_G(\zeta)$, and }
\P_{\mbox{\tiny{red}}}(\zeta,s) <\infty.
$$
\end{proposition}

\begin{proof}
Without loss of generality we may assume that $G$ has no elliptic elements. This follows mainly since $j(h, \zeta) = 1$ for any fixed point $\zeta$ of an
elliptic element $h\in G$, and since $\mathrm{Stab}_G(\zeta)$ has a subgroup of finite index with no elliptic elements. The details are  left to the reader.

If $j(h,\zeta)=1$, for each $h \in \mathrm{Stab}_G(\zeta)$, and if the reduced
horosperical Poincar\'e series converges, then $\mu_\zeta$ is a purely
atomic conformal measure. Here, $\mu_\zeta$ is defined as prior to
Lemma~\ref{bounded-parabolic}. This gives one direction of the equivalence.
%% This fact is independent of
%% whether $\mu_\zeta$ is a ending measure or not.
For the reverse direction, assume that there exists a purely atomic conformal measure
$\mu$ essentially supported on $G(\zeta)$.
Let us split the argument into the two cases: $\mathrm{Stab}_G(\zeta)$ is
 trivial and $\mathrm{Stab}_G(\zeta)$ is non-trivial.
If $\mathrm{Stab}_G(\zeta)$ is trivial,
then the $s$-conformality of $\mu$ implies that 
$\mu(\{g\zeta\})=j(g,\zeta)^s \mu(\{\zeta\})$. Thus,
$\mu = \sum_{g \in G} j(g,\zeta)^s \, \mu(\{\zeta\}) \, 
\hbox{\bf 1}_{g(\zeta)}$,
and therefore, since $\mu$ is a probability measure,
$$
\sum_{g \in G} j(g,\zeta)^s = 1/\mu(\{\zeta\}) < \infty.
$$
If $\mathrm{Stab}_G(\zeta)$ is non-trivial, then $\zeta$ is either a parabolic
or a loxodromic fixed point.
If $\zeta$ is a loxodromic fixed point of $G$, then
a conformal measure can not an atom at $\zeta$. This follows, since 
$j(h, \zeta) \neq 1$ for each loxodromic element $h \in G$ fixing
$\zeta$.
Therefore, $\zeta$ has to be a parabolic fixed point. In this situation we have that $j(h,\zeta)=1$, for each $h \in \mathrm{Stab}_G(\zeta)$.
Argueing similar as in the trivial stabiliser case,  it  follows  that  the reduced horospherical Poincar\'e series converges.
\end{proof}

Combining Theorem~\ref{basic1} and Proposition~\ref{pointmass}, 
the statement of Theorem~\ref{basic1} can be extended to the case 
in which $\mathrm{Stab}_G(\zeta)$ is non-trivial.
The proof is straightforward and left to the reader.

\begin{theorem}
\label{general}
%% Suppose that the Poincar\'e series $P(z,s)$
%% of a Kleinian group $G$ converges at the exponent $s \geq \delta(G)$.
For a Kleinian group $G$,
let $\left(z_{n}\right)$ be an ending sequence converging to
$\zeta \in D(G)$, and let $\mu$ be the associated $s$-conformal
ending measure. Then $\mu$ has an atom at $\zeta$ if and only if
$$
\mbox{$j(h,\zeta) = 1$ for each $h \in \mathrm{Stab}_G(\zeta)$, and }
  \P_{\mbox{\tiny{red}}}(\zeta,s)<\infty.
$$
In case the latter is satisfied, we then in particular have that
$\mu$ coincides with the purely atomic $s$-conformal measure
$$
\mu_\zeta = \frac{1}{\P_{\mbox{\tiny{red}}}(\zeta,s)} \sum_{g \in G/\mathrm{Stab}_G(\zeta)}
j(g,\zeta)^s \;
\hbox{\bf 1}_{g(\zeta)}
$$
which is essentially supported on $G(\zeta)$.
\end{theorem}

Let us also briefly mention another way to produce purely atomic conformal
ending measures. For this, let us assume that the end limit set associated
to some end of a hyperbolic manifold $M_G=\D/G$ is countable,
for some Kleinian group $G$. Here, the notions of end and end limit set
are adopted from \cite{AFT} (see also Section~\ref{pre}).
Theorem~4.4 in \cite{AFT} then ensures that the type of ending measure
constructed in the previous section is essentially supported on the
corresponding end limit set, and hence, this measure is necessarily
purely atomic. Recall the basic definitions concerning ends of hyperbolic manifolds given in Section~\ref{pre}, in particular the notion of end limit set of $G$ associated with some end $E$ of $M_G$.

\begin{proposition}
\label{alternative-atomic}
Let $G$ be a Kleinian group such that $M_G$ has an end $E$.
Furthermore, assume that the end limit set associated with $E$
is countable.
Then every conformal ending measure obtained from any ending
sequence converging to any end limit point of any end
$E$ is purely atomic.
\end{proposition}

Bishop \cite{bi} improved a classical result of Beardon and Maskit
\cite{BeMa}, by showing that a
Kleinian group $G$
acting on hyperbolic 3-space is geometrically finite if and only if
$L(G) \setminus L_r(G)$ is countable (possibly empty). It is not
difficult to see that this result continues to hold for Fuchsian groups
acting on hyperbolic 2-space.
Then, a combination of this result of Bishop and
Proposition~\ref{alternative-atomic} strongly supports
the following conjecture.
(Note that this conjecture is certainly true in two and three dimensions.)

\begin{conjecture*}
The end limit set of an end is countable if and only if it consists of
finitely many orbits of parabolic fixed points.
\end{conjecture*}

\begin{remark}
Note that if the latter conjecture holds, then this
would imply that
the class of examples mentioned prior to Proposition~\ref{pointmass}
represents the only class of purely atomic conformal ending measures
emerging from Proposition~\ref{alternative-atomic}.
\end{remark}

\section{Conformal
ending measures and the big horospherical limit set}
\label{big-horo-atoms}

In this section we show that if $\zeta$ is not fixed by any
parabolic element of a Kleinian group $G$ and if $\zeta$ is
an atom of some conformal ending measure for
$G$, then $\zeta$ lies in the dissipative part of the action of $G$ on
$\S$. Moreover, we will see that if the dissipative part of $L(G)$ has positive mass for some conformal measure, then there exists an atomic conformal ending measure for $G$.

Recall that a Kleinian group $G$ is said to act {\em dissipatively} on some measurable $G$-invariant set $A$, if there exists a measurable
set $E \subset A$ such that $E \cap G(\zeta)$ is a singleton, for each 
$\zeta \in A$. Note that the dissipative
action is defined independently of conformal measures.
Let $L_H(G)$ be the big horospherical limit set of $G$ and consider
the complement $\S \sm L_H(G)$,  which is clearly $G$-invariant and
 contained in the big Dirichlet set $D(G)=\Omega(G) \cup L_J(G)$.
In \cite{Tu} Tukia showed that $G$ acts dissipatively on $\S \sm L_H(G)$.
Moreover, in \cite{Tu} it was also shown that if $G$ acts dissipatively
on a $G$-invariant set $A$, then $A$ is contained in $\S \sm L_H(G)$
up to parabolic fixed points of $G$ and up to sets of measure zero, for any conformal measure for $G$.
In this sense, we can say that $\S \sm L_H(G)$ is the maximal subset of 
$\S$ on which $G$ acts dissipatively, and therefore we can say that 
$\S \sm L_H(G)$ represents the dissipative part of the action 
of $G$ on $\S$.

\begin{proposition}
\label{prep1}
For a Kleinian group $G$ we have that if $\zeta \in L_H(G)$
is not a parabolic fixed point of $G$, then $\zeta$ is not an atom
of $\mu$, for any $s$-conformal
measure $\mu$ for $G$.
\end{proposition}

\begin{proof}
If $\zeta \in L_H(G)$, then there is a sequence  $\left(g_n^{-1}\right)$ 
of elements $g_n^{-1}\in G$ such that $\left(g_n^{-1}(0)\right)$ converges to $\zeta$ within some horoball $H_\zeta(c)$.
In this case, we have $j(g_n,\zeta) >c$, for some $c >0$ and for all
$n \in \N$. Combining this with Proposition~\ref{pointmass}, 
it follows that $\zeta$ can be an atom only if
$\mathrm{Stab}_G(\zeta)$ is infinite and  if $j(h,\zeta) = 1$, for every
$h \in \mathrm{Stab}_G(\zeta)$.
However, this is only possible if $\zeta$ is  a parabolic fixed point
of~$G$.
\end{proof}

\begin{remark}
Note that the intersection $L_J(G) \cap L_H(G)$ can contain a point
$\zeta$ which is not a parabolic fixed point.
In this case, the horospherical Poincar\'e series
$\P(\zeta,s)$ at $\zeta$ diverges for every exponent
$s \in \Rp$. Indeed, if it would converge, then Theorem~\ref{basic1}
guarantees the existence of
a purely atomic $s$-conformal ending measure $\mu$ with an atom at
$\zeta$, contradicting Proposition~\ref{prep1}.
An example of such a point $\zeta \in L_J(G) \cap L_H(G)$ can be given
by taking a normal subgroup $G$ of some Kleinian group $\Gamma$.
By choosing a suitable hyperbolic element $\gamma \in \Gamma$
with $G \cap \langle \gamma \rangle=\{id.\}$, we can guarantee on the one hand  that the fixed
point $\zeta$ of $\gamma$ is contained in the boundary at infinity of some
Dirichlet fundamental domain of $G$. On the other hand, by a result of   \cite{Mat} we have that  $\zeta \in L_H(G)$.
\end{remark}

The following proposition shows that if there exists a conformal
measure which gives positive mass to the dissipative part, then
there also exists a purely atomic conformal ending measure.
This result relies on Tukia's results mentioned above.

\begin{proposition}
\label{prep2}
If there exists a $s$-conformal measure $\mu$ for $G$ such that
$\mu(A \sm L_H(G))>0$, for some $G$-invariant measurable set $A$, then
there exists a purely atomic $s$-conformal ending measure $\nu$ for $G$
essentially supported on $A$ such that each atom of $\nu$ is not fixed
by any parabolic element of $G$.
\end{proposition}

\begin{proof}
Let $A$ be $G$-invariant and  measurable such that $\mu(A \sm L_H(G))>0$.
By \cite{Tu}, we then have that there exists a measurable set
$E \subset A \sm L_H(G)$ such that $G(\zeta)$
has exactly one point in $E$, for each $\zeta \in A \sm L_H(G)$.
Using $\mu(A \sm L_H(G))>0$, it follows that $\mu(E)>0$.
Define a measurable subset $E_{n}$, for each positive integer $n$, by
$$
E_n=\{\zeta \in E : \#\mathrm{Stab}_G(\zeta) \leq n\}.
$$
Since $E=\bigcup_{n=1}^\infty E_n$ and $\mu(E)>0$,
there exists some $n$ such that $\mu(E_n)>0$. Hence,
$$
1 \geq \mu(A \sm L_H(G)) \geq \mu \left( \bigcup_{g \in G}g(E_n) \right)
\geq \frac{1}{n} \sum_{g \in G} \mu(g(E_n)).
$$
Since $\mu$ is an $s$-conformal measure for $G$, the final term
in the latter chain of inequalities is
equal to
$$
\frac{1}{n}\sum_{g \in G} \int_{E_n} j(g,\zeta)^s d\mu(\zeta) =
\frac{1}{n}\int_{E_n} \P(\zeta,s) d\mu(\zeta).
$$
This implies that $\P(\zeta,s)$ converges, for
$\mu$-almost every $\zeta \in E_n$. In particular, there exists
some $\xi \in E_n$ such that
$\P_{\mbox{\tiny{red}}}(\xi,s)<\infty$.
Since $\xi \in D(G)$, Theorem~\ref{general} shows that the measure
$$
\nu=\frac{1}{\P_{\mbox{\tiny{red}}}(\xi,s)}\sum_{g \in G/\mathrm{Stab}_G(\xi)}j(g,\xi)^s
{\bf 1}_{g(\xi)}
$$
is a purely atomic $s$-conformal ending measure.
In particular, since $\mathrm{Stab}_G(\xi)$ is finite, we have that $\xi$ is not fixed by any parabolic
element of $G$.
\end{proof}

We now summarize the two previous propositions in the following theorem.

\begin{theorem}
\label{thm5}
For a Kleinian group $G$
and for $s \geq \delta(G)$, the following are equivalent.
\begin{itemize}
\item[$(\rm i)$] Every $s$-conformal ending measure for $G$ supported on
$L(G)$ is essentially supported on $L_H(G)$.
\item[$(\rm ii)$] Every $s$-conformal ending measure for $G$ supported
on $L(G)$ can have atoms only at parabolic fixed points of $G$.
\end{itemize}
\end{theorem}

\begin{proof}
Let $\zeta \in L(G)$ be given such that $\zeta$ is not a parabolic
fixed point of $G$. Suppose that there exists a purely atomic
$s$-conformal ending measure $\mu$ for $G$
essentially supported on $G(\zeta)$.
By Proposition~\ref{prep1}, we then have that $\zeta$ does not belong
to $L_H(G)$, and hence, $\mu$ is essentially supported on
$L(G) \sm L_H(G)$. This shows that (i) implies (ii).

For the reverse implication, suppose that there is an $s$-conformal
ending measure $\mu$ for $G$ such that $\mu(L(G) \sm L_H(G))>0$.
By  applying Proposition~\ref{prep2}  to the situation where $A=L(G)$,
it follows that there exists a purely atomic $s$-conformal ending measure
for $G$ supported on $L(G)$ whose atoms are not parabolic fixed points
of~$G$.
\end{proof}

\section{Examples of conformal ending measures}
\label{examples}

In this section we discuss three non-trivial examples of purely atomic
as well as non-atomic conformal ending measures. The first example
considers infinitely generated Schottky groups which admit purely
atomic conformal ending measures with atoms at J{\o}rgensen points.
The second example looks at normal subgroups of finitely generated
Schottky groups. We show that for groups of this type there can be an
uncountable family of mutually singular non-atomic conformal ending
measures. The third example considers a Kleinian group $\Gamma$
which admits a purely atomic $\delta(\Gamma)$-conformal ending measure
at a J{\o}rgensen point  that is not a parabolic fixed point.

\bigskip
\noindent
{\bf Example 1.} \\
{\em Atomic measures at J{\o}rgensen points of
infinitely generated Schottky groups.}

We consider a particular class of infinitely generated classical
Schottky groups (see also \cite{SU07}).
In order to define these groups, let us fix some
increasing function $\phi: \N \to \Rp$ such that
$\phi$ is compatible with the following inductive construction
of the  set
$
{\cal C}={\cal C}(\phi)=\{C_{n}^\epsilon : \epsilon \in \{+.-\},n \in \N\}
$
of pairwise disjoint open $N$-dimensional discs  $C_{n}^{\pm}$ in
$\S$ with radii $r_{n}^{\pm}$.
For ease of exposition, let us assume that
$ r_{n}:=r_{n}^+=r_{n}^- <1$, for all $n \in \N$.
With $C_{n}^{\pm,\phi}(\supset C_{n}^{\pm})$ denoting  the
disc concentric to $C_{n}^\pm$ of radius $\phi(n) r_{n}$, we now choose
the elements in $\cal C$ by way of induction as follows.
Let $C_{1}^+$ and $C_{1}^-$ be chosen such that
$C_{1}^{+,\phi} \cap C_{1}^{-,\phi} = \emptyset$.
For the inductive step assume that
$C_{1}^+,C_{1}^-,\ldots,C_{n}^+,C_{n}^-$ have been constructed, for
some $n \in  \N$.
Then choose $C_{n+1}^+$ and $C_{n+1}^-$ such that
\[
C_{n+1}^{+,\phi} \cap C_{n+1}^{-,\phi} = \emptyset
\, \hbox{ and } \, C_{n+1}^{\epsilon,\phi}
\cap C_{k}^{\epsilon,\phi} = \emptyset \, \hbox{ for all } k \in
\{1,\ldots,n\},\, \epsilon \in \{+,-\}.
\]
With these discs at hand,   we then define
hyperbolic transformations $g_n^\epsilon\in \Con(N)$, for each $n \in \N$
and $\epsilon \in \{+,-\}$, such that $g_n^-=(g_n^+)^{-1}$ and
\[
g_{n}^+ \left( \Ext  (C_{n}^+) \right) = \Int (C_{n}^-)
\, \hbox{ and }  \,  g_{n}^{-} \left( \Ext  (C_{n}^-) \right) =
\Int(C_{n}^+).
\]
Here, $\Int$ and $\Ext$ denote respectively the interior and exterior in
the topology of $\S$.
Then let $G^{\phi}$ be
the group generated by the set
$$
G_{1}^{\phi} = \{g_{n}^\epsilon: \epsilon \in \{+,-\}, n \in \N\}.
$$
One easily verifies that $\bigcap_{C\in {\cal C}}\Ext (C)$ is a
fundamental domain for the action of $G^{\phi}$ on
$\S$ (cf. \cite{Mas}, Proposition VIII.A.4), which shows in particular
that $G^{\phi}$ is a Kleinian group.
Note that the set $L_{J}(G^{\phi})$ of J{\o}rgensen limit points
%is equal to
contains the accumulation points of $\{C\in {\cal C}\}$.
%which are contained in at most finitely many elements of ${\cal C}$
Then let $L_{J}'(G^{\phi})$ denote the subset of $L_{J}(G^{\phi})$ given by
\[
L_{J}'(G^{\phi}) := L_{J}(G^{\phi}) \cap  \bigcap_{C_{n}^{\epsilon,\phi}
\in
{\cal C}} \Ext (C_{n}^{\epsilon,\phi}).
\]
Clearly, by construction we have that $L_{J}'(G^{\phi}) \neq \emptyset$.

\begin{proposition}
Let $s \in \R^+$ be
given, and let $\phi_{s}:\N \to [2,\infty)$
be an increasing function such that
$ \sum _{n\in \N} \left(4/ \phi_{s}(n)\right)^{2s}<1/2$.
We then have that the horospherical Poincar\'e series for
$G^{\phi_{s}}$ at $\zeta$
converges at the exponent $s$,
for each $\zeta \in L_{J}'(G^{\phi_{s}})$. Moreover,
the essential support of the associated $s$-conformal
ending measure $\mu_{\zeta}$ is equal to $G^{\phi_{s}}(\zeta)$.
\end{proposition}

\begin{proof}
First note that for each
$\zeta \in \Ext (C_{n}^{\epsilon,\phi_{s}})$, $g_{n}^\epsilon \in
G_{1}^{\phi_{s}}$
and $\epsilon \in \{+,-\}$, we have that
\[ |\zeta-(g_{n}^\epsilon)^{-1}(0)| \geq \phi_{s}(n) r_{n} -r_{n} =
(\phi_{s}(n)-1)r_{n}.\]
Using this together with $\phi_{s}(n) \geq 2$ and the fact that
$1-|g_{n}^\epsilon(0)| \leq 2 r_{n}^{2}$,
it follows,  for each $\zeta \in \Ext (C_{n}^{\epsilon,\phi_{s}})$,
\[
j(g_{n}^\epsilon,\zeta) =
\frac{1-|g_{n}^\epsilon(0)|^{2}}{|(g_{n}^\epsilon)^{-1}(0) - \zeta|^{2}}
\leq
\frac{2 (1-|g_{n}^\epsilon(0)|)}{(\phi_{s}(n) -1)^{2}  r_{n}^{2}}
\leq
\frac{4}{(\phi_{s}(n)-1)^{2}}
\leq
\left( \frac{4}{\phi_{s}(n)}\right)^{2}.
\]

Then note that every element $\gamma \in G^{\phi_{s}}$ is uniquely
represented by a reduced word
$g_{n_1(\gamma)}^{\epsilon_1} \cdots g_{n_k(\gamma)}^{\epsilon_k}$
of elements in $G_{1}^{\phi_{s}}$.
As usual, $k$ will be referred to as the word length of $\gamma$.
With $G^{\phi_{s}}_k$ denoting the set of  elements of $G^{\phi_{s}}$
of word length $k \in \N$, we then obtain
for  $\zeta \in L_{J}'(G^{\phi_{s}})$, using the chain rule,
\begin{eqnarray*}
\sum_{\gamma \in G^{\phi_{s}} \setminus \{id.\}} j(\gamma,\zeta)^{s}
&=& \sum_{k  \in \N}
\sum_{\gamma \in G^{\phi_{s}}_{k}} j(g_{n_1(\gamma)}^{\epsilon_1} \cdots
g_{n_k(\gamma)}^{\epsilon_k},\zeta)^{s} \\
& = &
\sum_{k \in \N}  \sum_{\gamma \in G^{\phi_{s}}_{k}}
\prod_{i=1}^{k} j(g_{n_i(\gamma)}^{\epsilon_i},g_{n_{i+1}
(\gamma)}^{\epsilon_{i+1}} \cdots g_{n_k(\gamma)}^{\epsilon_k}
(\zeta))^{s} \\
& \leq &
\sum_{k \in \N}  \sum_{\gamma \in G^{\phi_{s}}_{k}} \prod_{i=1}^{k}
\left( \frac{4}{\phi_{s}(n_i(\gamma))} \right)^{2s} \\
& \leq &
\sum_{k \in \N} \left[2 \sum_{n \in \N}
\left( \frac{4}{\phi_{s}(n)}\right)^{2s}\right]^{k}
 < \infty.
\end{eqnarray*}
This shows that the horospherical Poincar\'e series for $G^{\phi_{s}}$
at $\zeta$ converges at the exponent $s$.
Applying Theorem~\ref{basic1} to the associated $s$-conformal ending
measure $\mu_{\zeta}$, finishes the proof of the proposition.
\end{proof}

\bigskip

\noindent
{\bf Example 2.} \\
{\em Non-atomic conformal ending measures for
normal subgroups of Schottky groups.}

Here we consider a class  of infinitely generated Kleinian
groups which has already been studied in \cite{FS}.
Let $G = G_1 * G_2$ denote the Schottky group obtained
as the free product of the two finitely generated non-elementary
Schottky groups $G_1$ and $G_2$. Assume that $\delta(G)>N/2$,
and that $G_1$ and $G_2$ have Dirichlet fundamental domains $F_1$ and
$F_2$  at $0 \in \D$ such that
$(\D \setminus F_1) \cap (\D \setminus F_2) = \emptyset$.
Then $G$ is isomorphic to the semidirect product $\Gamma \rtimes G_2$,
where $\Gamma$ denotes the kernel of the canonical
group homomorpism from $G$ onto $G_{2}$.
By the splitting lemma for groups, this is equivalent to the
existence of a short exact sequence of group homomorphisms
$$
1 \to \Gamma \to G \to G_2 \to 1.
$$
Then note that, since  $\Gamma$ is normal in $G$,
we have that $L(G)=L(\Gamma)$. Moreover, by a result
of Brooks \cite{Br}, we have  that $\delta(\Gamma) < \delta(G)$.
Finally, note that by construction we have that
$L(G_2) \subset L_J(\Gamma)$.

\begin{proposition}\label{mut-sing}
For $G, G_{1},G_{2}$ and $\Gamma$
given as above, the following hold:
\begin{itemize}
\item[$(\rm i)$]  $\Gamma$ is of convergence type.

\item[$(\rm ii)$] The $\delta(\Gamma)$-conformal ending
measure $\mu_{\zeta}$ for $\Gamma$ has no atoms,
for each $\zeta \in L_J(\Gamma)$.
\item[$(\rm iii)$]  If $\mu_{\zeta_{1}}$ and $\mu_{\zeta_{2}}$
are two $\delta(\Gamma)$-conformal ending measures for $\Gamma$,
for distinct  $\zeta_{1}, \zeta_2 \in L(G_2)$, then $\mu_{\zeta_{1}}$ and $\mu_{\zeta_{2}}$ are mutually singular.
\end{itemize}
\end{proposition}

\begin{proof}
For the proof of $(\rm i)$, note that Corollary~4.3 in \cite{MY} states
that if $\Gamma'$ is a Kleinian group of divergence type which is normal
in some other Kleinian group $G'$, then we have that
$\delta(\Gamma')=\delta(G')$ and that $G'$ is of $\delta(G')$-divergence
type.
Therefore, assuming that  $\Gamma$ is of
$\delta(\Gamma)$-divergence type would imply that
$\delta(\Gamma)=\delta(G)$, and this is clearly not the case.
This finishes the proof of $(\rm i)$.

For the proof of $(\rm ii)$, note that Theorem~6 in \cite{Mat} states
that if $\Gamma'$ is a non-trivial normal subgroup of a Kleinian
group $G'$, then $L_r(G') \subset L_H(\Gamma')$.
Hence for our situation here, where $\Gamma$ is normal
in $G$ and where $G$ is  a Schottky group,
it follows that $L(\Gamma)=L(G)=L_r(G)\subset L_H(\Gamma)$
and hence, we can deduce
$L(\Gamma)=L_H(\Gamma)$. An application of Proposition~4.1 in
Section~\ref{big-horo-atoms} then gives  that any
$\delta(\Gamma)$-conformal measure for $\Gamma$ is atomless
(note that we have by construction that $\Gamma$ does not have
parabolic elements). This completes the proof of $(\rm ii)$.

Finally, in order to prove $(\rm iii)$, let $\zeta_{1},  \zeta_2 \in
L(G_2)$ be given such that  $\zeta_1 \neq \zeta_2$. Since $G_2$ is a
finitely generated Schottky group, it is not  difficult to see that
there are hyperplanes in $\D$, denoted $H_1$ and $H_2$, which are at
the boundary of some image of the Dirichlet fundamental domain $F_2$
of $G_2$, such that the following holds. First, $H_1$ separates $\D$
into halfspaces one of which has $\zeta_1$ at infinity, and the other
contains $H_2$ and has $\zeta_2$ at infinity; secondly, the analogous
condition holds for $H_2$. By construction of $\Gamma$, it now follows
that $H_1$ and $H_2$, when projected to
$\overline M_\Gamma=(\D \cup \Omega(\Gamma))/\Gamma$, separate distinct
`ends with boundary' as introduced in \cite{AFT}, Subsection 4.3.
Since Theorem~4.4 in \cite{AFT} holds for such ends with boundary as
well, we can conclude that the measures $\mu_{\zeta_{1}}$ and
$\mu_{\zeta_{2}}$ have distinct essential supports and thus are
mutually singular.
\end{proof}

\bigskip
\noindent
{\bf Example 3.} \\
{\em Atomic $\delta(\Gamma)$-conformal ending measures
at  limit points which are not fixed points.}

We construct a Kleinian group $\Gamma$ for which
there is a purely atomic $\delta(\Gamma)$-conformal ending measure
at a J{\o}rgensen point of $\Gamma$ that is not a parabolic fixed point.
Similar as in Example~2, let $G = G_1 * G_2$ such that we have
an exact sequence of  group homomorphisms
$$
1 \to \Gamma \to G \to G_2 \to 1.
$$
This time  $G_{2}$ denotes a parabolic
Abelian subgroup of $\Con(N)$ with fixed point equal to $\zeta$, and $G_1 < \Con(N)$ is
a non-elementary Kleinian group
with $\Omega(G_1) \neq \emptyset$. As in the previous example, let us assume that
$G_1$ and $G_2$ have Dirichlet fundamental domains $F_1$ and
$F_2$  at $0 \in \D$
such that
$(\D \setminus F_1) \cap (\D \setminus F_2) = \emptyset$.
Moreover, we assume that $\delta(G)>N/2$ and  that
$G $ is  of convergence type. Clearly,  the latter condition is
 always satisfied if we for instance choose  $G_1$ such that
$\delta(G_1)=N$.
Then note that, since $G/\Gamma$ is isomorphic to the
amenable group $G_2$,
an application of a  result of  \cite{Br}
gives that $\delta(\Gamma) = \delta(G)$. Also,
note that  $\zeta$  is a bounded parabolic fixed point of
$G$, whereas $\zeta$ is a J{\o}rgensen limit point of $\Gamma$
which is not fixed by any non-trivial element of $\Gamma$.

\begin{proposition}
\label{final}
For $G, G_{1},G_{2},\Gamma$  and $\zeta$
given as above, we have that the $\delta(\Gamma)$-conformal ending
measure $\mu_{\zeta}$ for $\Gamma$ is essentially supported on
$\Gamma(\zeta)$. In particular, we hence have that $\mu_{\zeta}$ is purely atomic.
\end{proposition}

\begin{proof}
By Lemma~\ref{bounded-parabolic},  we have that the reduced horospherical
Poincar\'e series
$$
\sum_{g \in G/\mathrm{Stab}_G(\zeta)} j(g,\zeta)^{s}
$$
converges for $s$ equal to the exponent $\delta(G)$.
Recall that $\delta(G)=\delta(\Gamma)$, that
$\mathrm{Stab}_G(\zeta)=G_2$, and that one of our assumptions is that
$G$ is of convergence type. We now show that
the horospherical Poincar\'e series
$\sum_{\gamma \in \Gamma} j(\gamma,\zeta)^{\delta(\Gamma)}$
for $\Gamma$ at the exponent $\delta(\Gamma)$ converges. Indeed, since
$G=\Gamma \rtimes G_2$, we can take $\Gamma$ as a system of
representatives of $G/\mathrm{Stab}_G(\zeta)$.
With this choice, we then have
$$
\sum_{g \in G/\mathrm{Stab}_G(\zeta)} j(g,\zeta)^{\delta(G)}
=\sum_{\gamma \in \Gamma} j(\gamma,\zeta)^{\delta(\Gamma)}.
$$
Therefore, the horospherical Poincar\'e series for $\Gamma$ at the 
exponent $\delta(\Gamma)$ converges, and hence we can apply 
Theorem~\ref{general}, which then gives the existence of a purely atomic
$\delta(\Gamma)$-conformal ending measure for $\Gamma$
which is essentially supported on $\Gamma(\zeta)$.

\end{proof}

\end{document}